\documentclass[11pt, dvipsnames]{amsart}
\usepackage{amsthm}

\usepackage[english]{babel}
\usepackage{tikz}
\usepackage[shortlabels]{enumitem}
\usepackage{url}
\usepackage{hyperref}
\usepackage{fullpage}

\theoremstyle{plain}
\newtheorem{thm}{Theorem}[section]
\newtheorem{prop}[thm]{Proposition}
\newtheorem{lemma}[thm]{Lemma}
\newtheorem{cor}[thm]{Corollary}
\newtheorem{mainthm}{Main Theorem}
\theoremstyle{definition}
\newtheorem{df}[thm]{Definition}

\newtheorem*{notn*}{Notation}

\newtheorem{remark}[thm]{Remark}

\DeclareMathOperator{\End}{End}

\def\Q{\mathbb{Q}}
\def\Z{\mathbb{Z}}
\def\F{\mathbb{F}}

\newcommand{\cF}{\mathcal{F}}
\newcommand{\cO}{\mathcal{O}}
\newcommand{\cC}{\mathcal{C}}

\renewcommand{\bar}{\overline}
\usepackage{hyperref}

\title[Every finite abelian group is the group of rational points]{Every finite abelian group is the group of rational points of an ordinary abelian variety over $\F_2$, $\F_3$ and $\F_5$}
\date{\today}

\author{Stefano Marseglia}
\address{Mathematical Institute, Utrecht University, P.O. Box 80010, 3508 TA, Utrecht, The Netherlands}
\email{s.marseglia@uu.nl}
\thanks{Marseglia is supported by NWO grant VI.Veni.202.107.}

\author{Caleb Springer}
\address{Department of Mathematics, University College London, Gower Street, London WC1H 0AY,
UK\\
and The Heilbronn Insitute for Mathematical Research, Bristol, UK}
\email{c.springer@ucl.ac.uk}
\thanks{Springer is partially supported by National Science Foundation award {CNS-2001470} and by the Additional Funding Programme for Mathematical Sciences, delivered by EPSRC (EP/V521917/1) and the Heilbronn Institute for Mathematical Research.}



\begin{document}

\begin{abstract}
We show that every finite abelian group occurs as the group of rational points of an ordinary abelian variety over $\F_2$, $\F_3$ and $\F_5$.
We produce partial results for abelian varieties over a general finite field~$\F_q$.
In particular, we show that certain abelian groups cannot occur as groups of rational points of abelian varieties over  $\F_q$ when $q$ is large.
Finally, we show that every finite cyclic group arises as the group of rational points of infinitely many simple abelian varieties over~$\F_2$.
\end{abstract}

\maketitle

\section{Introduction}
    Recently, Howe and Kedlaya \cite{HK21} proved that every positive integer $m$ is the order of the group of rational points of an ordinary abelian variety over~$\F_2$.
    Shortly afterwards, Van Bommel, Costa, Li, Poonen and Smith \cite{vBCLPS21} extended the result to the finite fields $\F_3$ and $\F_5$.
    Similar results, with some exceptions, are obtained for the finite fields $\F_4$, $\F_7$ and,  when $m$ is large enough, for a general finite field $\F_q$.  In another direction, Kedlaya expanded the result of \cite{HK21} to prove that every positive integer is the order of the group of rational points of {\it infinitely many} (not necessarily ordinary) simple abelian varieties over $\F_2$ \cite[Theorem 1.1]{Ked21}.
    The goal of this paper is to strengthen these results from statements regarding cardinality to statements regarding groups.

    \subsection{Finite fields of small cardinality}
    For clarity, we start with a simplified statement of our first main result.     Recall that an abelian variety $A$ of dimension~$g$ over a finite field of characteristic $p$ is called {\it ordinary}, resp.~{\it almost ordinary}, if the $p$-torsion of~$A(\bar{\F}_p)$ consists of $p^g$ points, resp.~$p^{g-1}$.

\begin{mainthm}
\label{thm:main1}
    Let $G$ be a finite abelian group. Then the following statements hold.
\begin{enumerate}
    \item Let $k$ be one of the finite fields $\F_2$, $\F_3$ or $\F_5$. Then there is an ordinary abelian variety $A$ defined over $k$,  such that $A(k) \cong G$.
    \item Over $\F_4$, there is an abelian variety $B$ which is either ordinary or almost ordinary such that $B(\F_4)\cong G$.
\end{enumerate}

\end{mainthm}

    In fact, over $\F_2$, the proof of this result also provides a way to bound the dimension of the abelian variety $A$ appearing in the theorem.
    Moreover, the non-ordinary abelian varieties used over $\F_4$ can be described precisely.
    The version of the theorem including all details appears in Section~\ref{sec:theproof} as Theorem~\ref{thm:main1-full}.

    The outline of the proof of the first part of Main Theorem \ref{thm:main1} is as follows.
    To begin, we reduce to the case when the group $G$ is cyclic, and we focus our attention on isogeny classes with the key property of being {\it square-free}; see Definition \ref{def:square-free}.
    Square-free isogeny classes over prime fields $\F_p$ and square-free ordinary isogeny classes over any finite field~$\F_q$ are well-understood in terms of fractional ideals in \'etale algebras over~$\Q$; see Definition \ref{DelCS} and Proposition~\ref{prop:sqfreeord}.
    Using this description, we prove that in every such isogeny class there is an abelian variety with cyclic group of points; see Proposition~\ref{prop:cyclic}.
    Therefore, because the number of rational points on an abelian variety is an isogeny invariant \cite[Theorem~1.(c)]{Tate66}, we simply require a square-free isogeny class of abelian varieties with the correct number of points.
    The result of Howe and Kedlaya~\cite{HK21} provides the necessary isogeny classes in the case of $\F_2$, and the result of Van Bommel et al.~\cite{vBCLPS21} does the same for the remaining cases.
    See Theorems \ref{thm:HK} and \ref{thm:vBCLPS} for the precise statements of their results.
    This allows us to conclude the proof of the first part of Main Theorem \ref{thm:main1}.
    The second part requires a small modification of the argument since there is no ordinary abelian variety over $\F_4$ with $3$ rational points; see \cite[Remark 1.16]{vBCLPS21}.
    The argument shows that any finite cyclic group is the group of rational points of a square-free ordinary abelian variety over $\F_2$, $\F_3$ and $\F_5$; see Corollary~\ref{cor:Gcyclic}.
    The same is true over $\F_4$ and $\F_7$, with some exceptions; see again Corollary~\ref{cor:Gcyclic} and Corollary~\ref{cor:GcyclicF7}.

    \subsection{Finite fields of arbitrary cardinality}
    We pause to stress that Propositions~\ref{prop:sqfreeord} and~\ref{prop:cyclic} work over any finite field~$\F_q$.
    Since \cite{vBCLPS21} shows that, for any prime power $q$, every sufficiently large integer is the order of an abelian variety defined over $\F_q$,
this gives us the following result.
    \begin{mainthm} \label{thm:main2}
    Let $q$ be a prime power. If $m_1,\dots, m_r$ are integers satisfying $m_i \geq q^{3\sqrt{q} \log q}$ for all $i$, then the group $\Z/m_1\Z~\times\dots\times~\Z/m_r\Z$ is isomorphic to the group of rational points of an ordinary abelian variety over $\F_q$.
    \end{mainthm}

    Main Theorem \ref{thm:main2} is recalled in Section \ref{sec:larger} as Theorem \ref{thm:main-larger}, together with further discussion about general finite fields.
    For example, we show that abelian groups of small exponent, regardless of cardinality, never appear as a group of rational points for an abelian variety over $\F_q$ when $q$ is large; see Proposition \ref{prop:bad-groups}.

    \subsection{Infinitely many occurrences}
    We now return to the finite field $\F_2$.
    After proving that every positive integer is the order of the group of rational points of infinitely many simple abelian varieties over $\F_2$ \cite[Theorem~1.1]{Ked21},
    {Kedlaya} suggested that it would be interesting to show an analogous statement regarding groups.
    The results contained in Section \ref{sec:sqfree} allow us to immediately deduce that every finite cyclic group is the group of rational points of infinitely many simple abelian varieties over $\F_2$; see Proposition \ref{prop:infcyclic}.
    Using this, we can prove the following stronger statement, which will be recalled in Section \ref{sec:infinite} as Theorem \ref{thm:infmany}.

    \begin{mainthm}\label{thm:main3}
        For every $n\geq 1$, there is an infinite set of Weil $2$-polynomials $\{f_{n,j}(t) \}_{j\geq 1}$ which are pairwise coprime and enjoy the following property. For each $j$, every finite abelian group of cardinality~$n$ arises as the group of rational points of an abelian variety with Weil polynomial~$f_{n,j}(t)$.
    \end{mainthm}
    However, as noted by Kedlaya, a result of Kadets \cite{Kadets21} shows that these results are  impossible over $\F_q$ for {larger~$q$}.
    Indeed,  if $q >2$, then for each positive integer $m$ there are only finitely many simple abelian varieties with $m$ rational points.

\subsection{Related work}
We conclude this section by outlining additional relevant literature.
The groups of points of elliptic curves have been studied extensively, in particular in relation to their application to cryptography; see for example~\cite{Ruck87}, \cite[Theorem 3.3.15]{TVN07}, and~\cite{Vol88}.
The groups of points of abelian surfaces have been studied by Xing in~\cite{Xing94} and~\cite{Xing96}, Rybakov in~\cite{Rybakov12}, and by David et al.~in~\cite{Davidetal14}.
Such results were extended to  dimension three by Kotelnikova in \cite{Kotelnikova19}.
Giangreco-Maidana determined precisely when a given Weil polynomial corresponds to a cyclic isogeny class (Definition \ref{def:cyclic}) in \cite{Giangreco-Maidana19, Giangreco-Maidana20,Giangreco-Maidana21}.
Rybakov provided classification results for the group of points in~\cite{Rybakov10} and~\cite{Rybakov15} in terms of the Newton polygon of the characteristic polynomial of Frobenius.
The second author gave a classification in terms of the endomorphism ring of the abelian variety in~\cite{Springer21}.

\subsection*{Acknowledgements}
The authors thank Bjorn Poonen, Valentijn Karemaker and Kiran Kedlaya for their helpful comments.

\section{The square-free case}\label{sec:sqfree}

\begin{df}
\label{def:square-free}
    An isogeny class of abelian varieties over $\F_q$ is called {\it square-free} if the corresponding characteristic polynomial of Frobenius (also known as its \emph{Weil polynomial}) has no multiple complex roots.
    An abelian variety $A$ over $\F_q$ is called {\it square-free} if it belongs to a square-free isogeny class.
\end{df}

\begin{df}\label{DelCS}
    Let $A$ be an abelian variety over a finite field $\F_q$.
    We say that $A$ satisfies condition {\bf Ord} if it is ordinary.
     We say that $A$ satisfies condition {\bf CS} if 
     $q$ is a prime number and the Weil polynomial has no real roots; such abelian varieties were studied by Centeleghe and Stix in \cite{CentelegheStix15}.
\end{df}

Note that in \cite{HK21}, an isogeny class is called square-free if the isogeny decomposition has no repeated factors. Their definition differs from ours in general, so we provide the following lemma to disambiguate the use of terminology.  For example,  we do not call an elliptic curve over $\F_{p^2}$ with Weil polynomial $(x - p)^2$ square-free in this paper even though it is simple.

\begin{lemma}\label{lem:sqfreeequiv}
Let $A$ be an abelian variety over $\F_q$. The following are equivalent.
\begin{enumerate}[(1)]
		\item $A$ has square-free Weil polynomial, i.e. $A$ is square-free.
		\item The endomorphism algebra $\End_{\F_q}(A)\otimes \Q$ is commutative.
\end{enumerate}
Additionally, these conditions imply the following.
\begin{enumerate}[(1)]
		\item[(3)]  The isogeny decomposition of $A$ contains no repeated factors.
\end{enumerate}
Moreover, if $A$ satisfies {\bf Ord} or {\bf CS}, then all three conditions are equivalent.
In particular, if $A$ satisfies {\bf Ord} or {\bf CS}, then $A$ is simple if and only if its Weil polynomial is irreducible.
 \end{lemma}
  \begin{proof}
  The first equivalence is \cite[Theorem 2(c)]{Tate66}.
  The implication that the first two conditions imply the third follows from \cite[Theorem 1(b)]{Tate66}.
   For property {\bf Ord}, it follows from Honda-Tate theory that an ordinary isogeny class is simple if and only if its Weil polynomial is irreducible; see for example~\cite[Theorem 3.3]{Howe95}. For property {\bf CS}, note that the Weil polynomial of any simple abelian variety over a prime finite field $\F_p$ is irreducible, unless the Weil polynomial has a real root; see for example \cite[Theorem~6.1]{Wat69} or \cite[p.96]{Tate71}.
  \end{proof}
\begin{remark}\label{rmk:simpleCS}
    Over a prime field $\F_p$, there is only one simple isogeny class whose Weil polynomial has real roots, namely $(x^2-p)^2$.
\end{remark}

Square-free abelian varieties satisfying {\bf Ord} or {\bf CS} are well-understood in terms of fractional ideals, thanks to certain equivalences of categories proved by Deligne in \cite{Del69} and Centeleghe-Stix in \cite{CentelegheStix15}.
We summarize the relevant results in the following Proposition; see \cite[Cor.~4.4 and Cor.~4.7]{MarAbVar18} for the proofs.
Let $f_A$ be the characteristic polynomial of Frobenius for a square-free abelian variety $A$ satisfying {\bf Ord} or {\bf CS}.
Let~$K$ be the \'etale algebra generated by the Frobenius endomorphism $\pi$, that is,~$K=\Q[\pi]=\Q[x]/(f_A)$.
Denote by $\Z[\pi,\bar{\pi}]$ the order in $K$ generated by the Frobenius and Verschiebung endomorphisms.
\begin{prop}
\label{prop:sqfreeord}
There is an equivalence\footnote{In the {\bf CS} case, the functor is contravariant. This detail is not needed in this paper.}
 $\cF$ between
the category of abelian varieties isogenous to~$A$ (with $\F_q$-homomorphisms) and the category of fractional $\Z[\pi,\bar{\pi}]$-ideals in $K$ (with $\Z[\pi,\bar{\pi}]$-linear morphisms).
In particular, if~$\cF(B) = I$ then we have an isomorphism of finite abelian groups
\[ B(\F_q) \cong \frac{I}{(\pi -1 ) I}. \]
\end{prop}

   Before applying this proposition and concluding this section, we recall one more definition.

    \begin{df}
    \label{def:cyclic}
    An abelian variety $A$ over $\F_q$ is \emph{cyclic} if $A(\F_q)$ is a cyclic group.  An isogeny class is cyclic if every abelian variety in the isogeny class is cyclic.
     \end{df}

Using Proposition \ref{prop:sqfreeord}, we can prove the existence of cyclic abelian varieties within every isogeny class satisfying {\bf Ord} or {\bf CS}.  The following proposition can be viewed as a generalization of a result of Galbraith for elliptic curves \cite[Lemma 1]{Galbraith99}.

\begin{prop}
\label{prop:cyclic}
    Every square-free isogeny class over $\F_q$ satisfying {\bf Ord} or {\bf CS} contains a cyclic abelian variety
\end{prop}
\begin{proof}
    We fix a square-free isogeny class over $\F_q$  satisfying {\bf Ord} or {\bf CS}.
    By Proposition \ref{prop:sqfreeord} there exists an abelian variety $A$ in such an isogeny class such that $\cF(A) = \Z[\pi, \bar\pi]$.
	Observe that a natural surjective map
	$$
		 \varphi: \frac{\Z[x, y]}{(f_A(x),xy-q, x - 1)} \to \frac{\Z[\pi, \overline\pi]}{(\pi - 1)} 
	$$
	 is induced by mapping $x$  to $\pi$ and $y$ to $\overline\pi$ because $f_A(\pi) = 0$ and $\pi\overline\pi = q$.
	 The target of $\varphi$ is isomorphic to the group of rational points $A(\F_q)$ by Proposition~\ref{prop:sqfreeord}. To conclude the proof, we will show that the domain is a cyclic group of order $f_A(1) = \#A(\F_q)$, which implies that the natural surjective map is actually an isomorphism. We pause to note that the closely-related surjective map $\Z[x, y]/(f_A(x),xy-q) \to \Z[\pi, \overline\pi]$  is \emph{not} an isomorphism in general; see \cite[Theorem 11]{CentelegheStix15} for the correct description of $\Z[\pi, \overline\pi]$ which could be used in an alternate presentation of this proof.

    By the division algorithm we can write $f_A(x) = (x-1)p(x) + f_A(1)$, for some polynomial $p(x)$.
    This relation together with $y-q = (xy-q) - y(x-1)$
    shows that we have the following equality of ideals of $\Z[x,y]$:
    \[ (f_A(x),xy-q,x-1) = (f_A(1),x-1,y-q). \]
    Therefore the evaluation map $(x,y)\mapsto (1,q) $ induces an isomorphism
    \[
        \frac{\Z[x,y]}{(f_A(1),x-1,y-q)}
        \cong \frac{\Z}{(f_A(1))}.
    \]
    We conclude that $\varphi$ is an isomorphism and $A(\F_q)$ is a cyclic group.
\end{proof}
\begin{remark}\label{rem:Rybakov_cyclic}
We note that the same result can be deduced from \cite{Rybakov10}, and
in the simple case, the result can also be deduced from \cite[Theorem 1.3]{Springer21}.
 Our proof shows that we can choose the cyclic ordinary abelian variety $A$ to have endomorphism ring isomorphic to $\Z[\pi, \bar\pi]$. In fact, any abelian variety in the given isogeny class with minimal endomorphism ring is cyclic because $\Z[\pi,\bar\pi]$  is Gorenstein; see \cite[Theorem 11]{CentelegheStix15}.
\end{remark}

\section{Proof of Main Theorem \ref{thm:main1}}
\label{sec:theproof}
We now focus on abelian varieties defined over $\F_2$.
As indicated in the introduction, Howe and Kedlaya proved the following theorem.
\begin{thm}[Theorem 1, \cite{HK21}]
\label{thm:HK}
    Let $m > 0$ and $d > 2$ be integers with $m < (4/3)2^d + 1$.
    Then there is a square-free ordinary abelian variety $A$ over~$\F_2$ of dimension at most $d$ with $m = \#A(\F_2)$.
\end{thm}

Over $\F_3, \F_4$ and $\F_5$, Van Bommel et al.~proved the following result.
\begin{thm}[{\cite[Theorem 1.13(a), Remarks 1.16 and 1.18]{vBCLPS21}}]
\label{thm:vBCLPS}
    Let $m$ be a positive integer and $k$ be $\F_3,\F_4$ or $\F_5$.
    Then there is a square-free abelian variety $A$ over~$k$ with $m = \#A(k)$.
    One can choose $A$ to be ordinary except in the case $m=3$ and $k=\F_4$.
\end{thm}

We are now ready to complete the proof of Main Theorem \ref{thm:main1}.
\begin{thm}
\label{thm:main1-full}
Let
$$
	G = \Z/n_1\Z \times\dots \times \Z/n_r\Z
$$
 be a finite abelian group. The following statements hold.
\begin{itemize}
    \item Let $k$ be one of the finite fields $\F_2,\F_3,\F_5$. Then there is
    an ordinary abelian variety $A$ defined over $k$,  such that $A(k) \cong G$;
    \item There is an abelian variety $B$ over $\F_4$, such that $B(\F_4)\cong G$ and $B$ is either ordinary or of the form $B\cong B'\times E$ where $B'$ is ordinary and $E$ belongs to the unique isogeny class of supersingular elliptic curves over $\F_4$ with $3$ rational points; see the LMFDB label \cite[\href{https://www.lmfdb.org/Variety/Abelian/Fq/1/4/ac}{{1.4.ac}}]{lmfdb}.
    The variety $B$ can be taken to be ordinary unless $G$ is a $3$-group such that $G\cong (\Z/3\Z)^{n_1}\times \prod_{j > 1} (\Z/3^j\Z)^{n_j}$ for $n_1$ odd.
\end{itemize}
 Moreover, if $k=\F_2$ and $d_1,\dots, d_r$ are integers satisfying $n_j < (4/3)2^{d_j} + 1$ and~$d_j\geq 3$ for each $1\leq j\leq r$,
 then there is an ordinary abelian variety $A$ defined over $\F_2$ of dimension at most $d = d_1+\cdots +d_r$ such that~$A(\F_2) \cong G$.
\end{thm}

\begin{proof}
    Assume that $k$ is one of the finite fields $\F_2$, $\F_3$, or $\F_5$.
    We can immediately reduce to the case when $r = 1$, that is, when the group~$G$ is cyclic.
    By Theorems \ref{thm:HK} and \ref{thm:vBCLPS},
    there exists a square-free ordinary isogeny class over $k$ with $\vert G \vert$ rational points.
    By Proposition~\ref{prop:cyclic}, we have an abelian variety $A$ within this isogeny class with cyclic group of points, that is, we have $A(k) \cong G$.

    We deal now with the finite field $\F_4$.
    Consider the simple ordinary isogeny class of abelian surfaces over $\F_4$ with $9$ rational points given by the Weil polynomial $x^4-3x^3+7x^2-12x+16$; see the LMFDB label \cite[\href{https://www.lmfdb.org/Variety/Abelian/Fq/2/4/ad_h}{{2.4.ad\_h}}]{lmfdb}.
    Let $K=\Q(\pi)$ be the endomorphism algebra of this isogeny class and let $\cO_K$ be the maximal order of $K$.
    We observe that
    \[ \dfrac{\cO_K}{(\pi-1)\cO_K}\cong \dfrac{\Z}{3\Z}\times \dfrac{\Z}{3\Z}. \]
    Hence by Proposition \ref{prop:sqfreeord}, in this isogeny class there is an abelian surface $B_0$ with group of rational points isomorphic to $(\Z/3\Z)^2$.
    Let $E$ belong to the unique isogeny class of supersingular elliptic curves over $\F_4$ with $3$ rational points; see the LMFDB label \cite[\href{https://www.lmfdb.org/Variety/Abelian/Fq/1/4/ac}{{1.4.ac}}]{lmfdb}.

    Write the group $G$ as a product of cyclic groups
    \[
    	G \cong (\Z/3\Z)^{2e + \delta} \times (\Z/s_1\Z)\times \dots \times (\Z/s_f\Z)
    \]
    where $\delta\in \{0,1\}$, $e\geq 0$ and each $s_j$ is a prime power satisfying either $s_j = 2$ or $s_j  > 3$ for all~$j$.
    Using {Theorem~\ref{thm:vBCLPS}} with the method from above, we can find an ordinary abelian variety $B_1$ over $\F_4$ whose group of rational points is isomorphic to $(\Z/s_1\Z)\times \dots \times (\Z/s_f\Z)$.
    If $G$ is not a $3$-group then there exists an $s_j$ which is not divisible by $3$, say $s_1$.
    Again, the above construction using {Theorem~\ref{thm:vBCLPS}} gives us an ordinary abelian variety $B'_1$ over $\F_4$ whose group of rational points is isomorphic to $(\Z/3s_1\Z)\times \dots \times (\Z/s_f\Z)$.

    With this set up, we now distinguish three cases.
    If $\delta=0$ then we set $B=B_0^{e}\times B_1$.
    If $\delta=1$ and $G$ is not a $3$-group then set $B=B_0^{e}\times B'_1$.
    Finally, if $\delta=1$ and $G$ is a $3$-group then we set $B = E^\delta\times B_0^{e}\times B_1$.
    In all three cases, we have that $B(\F_4)\cong G$ by construction and the Chinese Remainder Theorem.
    In the first two cases, the variety $B$ is ordinary, while in the last one $B$ is almost ordinary.
\end{proof}

\begin{remark}
    We stress that Theorem \ref{thm:main1-full} does not exclude the existence of ordinary abelian varieties over $\F_4$ with group of points isomorphic to $(\Z/3\Z)^{n_1}\times \prod_{j > 1} (\Z/3^j\Z)^{n_j}$ for $n_1$ odd.
    Indeed, the LMFDB \cite{lmfdb} contains ordinary abelian varieties over $\F_4$ with group of points $\Z/3\Z\times \Z/9\Z$;
     for example, see the isogeny class with label
    \cite[\href{https://www.lmfdb.org/Variety/Abelian/Fq/2/4/b_f}{{2.4.b\_f}}]{lmfdb}.
    However, the LMFDB does not currently contain any ordinary abelian varieties over $\F_4$ with a group of rational points isomorphic to $(\Z/3\Z)^{2e+1}$ for $e\geq 1$. Nonexistence when $e = 0$ follows from \cite[Theorem 3.2]{Kadets21}; see \cite[Remark 1.16]{vBCLPS21}.
    Proving more general existence or nonexistence results will require additional understanding of which groups can occur as the group of rational points of an abelian variety in a given isogeny class.
\end{remark}

From the proof of Theorem \ref{thm:main1-full} we immediately deduce the following special case.
\begin{cor}
\label{cor:Gcyclic}
    Let $k$ be one of the finite fields $\F_2$, $\F_3$ or $\F_5$.
    Let $G$ be a finite cyclic abelian group.
    Then there are
    \begin{itemize}
    \item a square-free ordinary abelian variety $A$ over $k$ with $A(k)\cong G$; and
    \item a square-free abelian variety $B$ over $\F_4$ with $B(\F_4)\cong G$, which we can choose to be ordinary if $G\neq \Z/3\Z$.
    \end{itemize}
\end{cor}

In a similar fashion, using \cite[Theorem 1.13(a), Remarks 1.17 and 1.18]{vBCLPS21}, we can achieve a result analogous to Corollary \ref{cor:Gcyclic} for abelian varieties over the finite field $\F_7$.
\begin{cor}
\label{cor:GcyclicF7}
    Let $G$ be a finite cyclic abelian group with $\vert G \vert \not\in \{ 2, 8, 14, 16, 17, 73 \}$.
    Then there is a square-free ordinary abelian variety $A$ over $\F_7$ with~$A(\F_7)\cong G$.
\end{cor}
\begin{remark}
Using Proposition \ref{prop:cyclic} or Remark \ref{rem:Rybakov_cyclic}, one can construct square-free abelian varieties over $\F_7$, necessarily non-ordinary, with group of rational points isomorphic to $\Z/8\Z$ and $\Z/73\Z$.
\end{remark}

\section{Proof of Main Theorem \ref{thm:main2}}
\label{sec:larger}

For large $q$, the Weil bounds prohibit the existence of abelian varieties over $\F_q$ with a relatively small number of points; see \cite{Weil48} and \cite{Kadets21}.
However, in \cite{vBCLPS21} it is proven that:
\begin{thm}[{\cite[Theorem 1.13(b), Remarks 1.16 and 1.18]{vBCLPS21}}]
    \label{thm:sufflarge}
    For an arbitrary prime power~$q$, every integer $m \geq q^{3\sqrt{q} \log q}$ arises as $m = \#A(\F_q)$ for some ordinary square-free abelian variety $A$ over $\F_q$.
\end{thm}
Still, Theorem \ref{thm:sufflarge} does not imply that every {abelian group} of sufficiently large order arises as the group of points of an abelian variety over $\F_q$, as shown by the following proposition, which was suggested by Bjorn Poonen.
\begin{prop}
\label{prop:bad-groups}
    Let $m > 1$ be an integer and $q$ a power of a prime $p$. Suppose there exists an abelian variety $A$ over $\F_q$ such that $A(\F_q)$ is a group of exponent $m$. Then  we have the following.
    \begin{enumerate}
        \item Unconditionally, $q \leq (m+1)^2$;
        \item If $m$ is also a power of $p$, then we have $q\leq (\sqrt{m} + 1)^2$.
    \end{enumerate}
\end{prop}
\begin{proof}
    Suppose that $g$ is the dimension of $A$. Then, using the Weil bounds and the structure of torsion subgroups, we get
    $$
    	(\sqrt{q} - 1)^{2g}
    	\leq \#A(\F_q)
    	\leq \#A(\overline\F_q)[m]
    	\leq m^{2g}.
    $$
    Rearranging, we obtain $q \leq (m + 1)^2$.  The  stricter upper bound follows similarly by using the structure of the $p^e$-torsion of an abelian variety in characteristic~$p$.
\end{proof}

\begin{cor}
\label{cor:2r}
    If $q >9$ or $q = 8$, then the group $(\Z/2\Z)^r$ for $r \geq 1$ does not arise as the group of rational points for any abelian variety over $\F_q$.
\end{cor}

The following theorem, whose proof is exactly like that of Theorem \ref{thm:main1-full}, is the best that can be obtained with our current methods. The nonexistence results above show that restrictions on the sizes of the cyclic factors are necessary.
We observe that there are some finite abelian groups which are not outlawed by Proposition \ref{prop:bad-groups} but which are also not realized in Theorem \ref{thm:main-larger}.  For general $q$, determining which of these groups arise as the group of rational points of an ordinary abelian variety over $\F_q$ remains an open question for future research.

\begin{thm}
\label{thm:main-larger}
Let $q$ be a prime power.
If $m_1,\dots, m_r$ are
integers satisfying $m_i \geq q^{3\sqrt{q} \log q}$ for all $i$,
then the group $\Z/m_1\Z~\times\dots\times~\Z/m_r\Z$ is isomorphic to the group of rational points of an ordinary abelian variety over $\F_q$.
\end{thm}

In the previous section, we deduced that certain cyclic groups can be realized as the group of rational points of a square-free abelian variety over $\F_q$ for various small values of~$q$; see Corollaries~\ref{cor:Gcyclic} and~\ref{cor:GcyclicF7}.
For arbitrary $q$, we prove an analogous corollary concerning square-free abelian varieties with an explicit bound for the size of the group.
\begin{cor}
    \label{cor:cyclic_precise_bound}
    Fix a prime power $q$ and let $m$ be an integer satisfying $m \geq q^{3\sqrt{q} \log q}$.
    Then there is a square-free ordinary abelian variety over $\F_q$ whose group of rational points is isomorphic to~$\Z/m \Z$.
\end{cor}
Alternatively, we prove the following result concerning \emph{geometrically simple} abelian varieties.
As a trade-off for this stronger condition on the abelian variety, there is no effective lower bound for the size of the group in the following theorem.
\begin{thm}
\label{thm:GcyclicFq}
	Fix a prime power $q$ and let $n$ be sufficiently large with respect to $q$.  There is a geometrically simple ordinary abelian variety $A$ over $\F_q$ with~$A(\F_q)\cong \Z/n\Z$.
\end{thm}
\begin{proof}
    When $n$ is sufficiently large, another result \cite[Corollary 1.3]{vBCLPS21} from the same paper as Theorem \ref{thm:sufflarge} proves that there is a geometrically simple ordinary abelian variety $A_0$ over $\F_q$ with~$\vert A_0(\F_q)\vert = n$.  By Proposition \ref{prop:cyclic}, there is an abelian variety $A$ isogenous to $A_0$ which is cyclic, which concludes the theorem.
\end{proof}

  \section{Proof of Main Theorem \ref{thm:main3}}\label{sec:infinite}
The goal of this section is to prove Main Theorem \ref{thm:main3}.  As a foundation, we have the following result of Kedlaya regarding cardinality.

    \begin{thm}[Theorem 1.1, \cite{Ked21}]
    Every positive integer is equal to the order of the group of rational points of infinitely many simple abelian varieties (of various dimensions) over $\F_2$.
    \end{thm}

    Using this result together with Lemma \ref{lem:sqfreeequiv}, Remark \ref{rmk:simpleCS} and Proposition~\ref{prop:cyclic}, we can immediately prove the following result.

    \begin{prop}\label{prop:infcyclic}
    Let $n \geq 1$ be a positive integer. There are infinitely many simple abelian varieties $A$ over $\F_2$ with $A(\F_2)\cong \Z/n\Z$.
    \end{prop}

    When searching for examples of (possibly non-simple) abelian varieties $A$ over $\F_2$ such that $A(\F_2)$ is isomorphic to a given finite abelian group $G$, there is one way to ``cheat" to find an infinite number of examples: Find one example $A_0$ with $A_0(\F_2)\cong G$ using Theorem \ref{thm:main1-full}, then consider the infinite set $\{A_0\times B \mid \#B(\F_2) = 1\}$.  The fact that there are infinitely many abelian varieties $B$ over $\F_2$ with only one point is a special case of Kedlaya's theorem which was originally proven by Madan and Pal \cite{MP77}.
In order to prove the existence of genuinely interesting infinite storehouses of examples, we find examples with pairwise coprime Weil polynomials. This corresponds to finding examples whose isogeny classes share no simple factors in common. We now prove Main Theorem~\ref{thm:main3}.

    \begin{thm}\label{thm:infmany}
    For every $n\geq 1$, there is an infinite set of Weil $2$-polynomials $\{f_{n,j}(t) \}_{j\geq 1}$ which are pairwise coprime and enjoy the following property. For each $j$, every finite abelian group of cardinality~$n$ arises as the group of rational points of an abelian variety with Weil polynomial $f_{n,j}(t)$.
    \end{thm}
    \begin{proof}
    Write $n = \ell_1\dots \ell_r$, where $\ell_1, \dots, \ell_r$ are primes that are not necessarily distinct.
    By Proposition \ref{prop:infcyclic}, for each ${1\leq i\leq r}$, there are infinitely many simple abelian varieties $A_i$  over $\F_2$ such that $A_i(\F_2) \cong \Z/\ell_i\Z$.
    By Remark \ref{rmk:simpleCS}, there is only one simple isogeny class over $\F_2$ whose Weil polynomial has real roots; see \cite[\href{https://www.lmfdb.org/Variety/Abelian/Fq/2/2/a_ae}{{2.2.a\_ae}}]{lmfdb}.
    Hence, by Lemma \ref{lem:sqfreeequiv}, we can choose the $A_i$
    to have distinct irreducible Weil polynomial.
    Let $\cC$  be the isogeny class of $A_1 \times \dots \times A_r$.
    Note that $\cC$ is square-free by construction.

    Because each sub-product $\prod_{j\in S} A_j$ for $S\subset \{1,\dots, r\}$ is also square-free, we can apply Proposition \ref{prop:cyclic} to get a cyclic isogenous variety $A_S$ satisfying $A_S(\F_2) \cong \Z/(\prod_{j\in S}\ell_j)\Z$. In this way, we obtain all abelian groups of cardinality $n$ as the group of rational points of an abelian variety in~$\cC$.
  \end{proof}

  \begin{remark}
      With the currently technology we are not able to determine whether given an arbitrary finite abelian group $G$ there are infinitely many \emph{simple} abelian varieties over $\F_2$ with group of rational points isomorphic to $G$.
      A possible approach to this questions is to study which groups can occur in a given isogeny class.
  \end{remark}

\bibliographystyle{amsalpha}

\begin{thebibliography}{vBCL{\etalchar{+}}21}

\bibitem[vBCL{\etalchar{+}}21]{vBCLPS21}
Raymond van {Bommel}, Edgar {Costa}, Wanlin {Li}, Bjorn {Poonen}, and Alexander
  {Smith}, \emph{{Abelian varieties of prescribed order over finite fields}},
  arXiv e-prints (2021), arXiv:2106.13651.

\bibitem[CS15]{CentelegheStix15}
Tommaso~Giorgio Centeleghe and Jakob Stix, \emph{Categories of abelian
  varieties over finite fields, {I}: {A}belian varieties over
  {$\mathbb{F}_p$}}, Algebra Number Theory \textbf{9} (2015), no.~1, 225--265.
  \MR{3317765}

\bibitem[Del69]{Del69}
Pierre Deligne, \emph{Vari\'et\'es ab\'eliennes ordinaires sur un corps fini},
  Invent. Math. \textbf{8} (1969), 238--243. \MR{0254059}

\bibitem[DGS{\etalchar{+}}14]{Davidetal14}
Chantal David, Derek Garton, Zachary Scherr, Arul Shankar, Ethan Smith, and
  Lola Thompson, \emph{Abelian surfaces over finite fields with prescribed
  groups}, Bull. Lond. Math. Soc. \textbf{46} (2014), no.~4, 779--792.
  \MR{3239616}

\bibitem[Gal99]{Galbraith99}
Steven~D. Galbraith, \emph{Constructing isogenies between elliptic curves over
  finite fields}, LMS J. Comput. Math. \textbf{2} (1999), 118--138.
  \MR{1728955}

\bibitem[GM19]{Giangreco-Maidana19}
Alejandro~J. Giangreco-Maidana, \emph{On the cyclicity of the rational points
  group of abelian varieties over finite fields}, Finite Fields Appl.
  \textbf{57} (2019), 139--155. \MR{3921286}

\bibitem[GM20]{Giangreco-Maidana20}
\bysame, \emph{Local cyclicity of isogeny classes of abelian varieties defined
  over finite fields}, Finite Fields Appl. \textbf{62} (2020), 101628, 9.
  \MR{4053144}

\bibitem[GM21]{Giangreco-Maidana21}
\bysame, \emph{Corrigendum to ``{L}ocal cyclicity of isogeny classes of abelian
  varieties defined over finite fields'' [{F}inite {F}ields {A}ppl. 62 (2020)
  101628]}, Finite Fields Appl. \textbf{69} (2021), 101703, 2. \MR{4183334}

\bibitem[HK21]{HK21}
Everett~W. Howe and Kiran~S. Kedlaya, \emph{Every positive integer is the order
  of an ordinary abelian variety over $\mathbb{F}_2$}, Research in Number
  Theory \textbf{7} (2021), no.~4, 59.

\bibitem[How95]{Howe95}
Everett~W. Howe, \emph{Principally polarized ordinary abelian varieties over
  finite fields}, Trans. Amer. Math. Soc. \textbf{347} (1995), no.~7,
  2361--2401. \MR{1297531}

\bibitem[Kad21]{Kadets21}
Borys Kadets, \emph{Estimates for the number of rational points on simple
  abelian varieties over finite fields}, Math. Z. \textbf{297} (2021), no.~1-2,
  465--473. \MR{4204701}

\bibitem[{Ked}21]{Ked21}
Kiran~S. {Kedlaya}, \emph{{Abelian varieties over $\mathbb{F}_2$ of prescribed
  order}}, arXiv e-prints (2021), arXiv:2107.12453.

\bibitem[Kot19]{Kotelnikova19}
Yulia Kotelnikova, \emph{Groups of points on abelian threefolds over finite
  fields}, Finite Fields Appl. \textbf{58} (2019), 177--199. \MR{3947815}

\bibitem[{LMF}21]{lmfdb}
The {LMFDB Collaboration}, \emph{The {L}-functions and modular forms database},
  \url{http://www.lmfdb.org}, 2021, [Online; accessed 20 July 2021].

\bibitem[Mar21]{MarAbVar18}
Stefano Marseglia, \emph{Computing square-free polarized abelian varieties over
  finite fields}, Math. Comp. \textbf{90} (2021), no.~328, 953--971.
  \MR{4194169}

\bibitem[MP77]{MP77}
Manohar~L. Madan and Sat Pal, \emph{Abelian varieties and a conjecture of {R}.
  {M}. {R}obinson}, J. Reine Angew. Math. \textbf{291} (1977), 78--91.
  \MR{439848}

\bibitem[R{\"{u}}c87]{Ruck87}
Hans-Georg R{\"{u}}ck, \emph{A note on elliptic curves over finite fields},
  Math. Comp. \textbf{49} (1987), no.~179, 301--304. \MR{890272}

\bibitem[Ryb10]{Rybakov10}
Sergey Rybakov, \emph{The groups of points on abelian varieties over finite
  fields}, Cent. Eur. J. Math. \textbf{8} (2010), no.~2, 282--288. \MR{2610753}

\bibitem[Ryb12]{Rybakov12}
\bysame, \emph{The groups of points on abelian surfaces over finite fields},
  Arithmetic, geometry, cryptography and coding theory, Contemp. Math., vol.
  574, Amer. Math. Soc., Providence, RI, 2012, pp.~151--158. \MR{2961407}

\bibitem[Ryb15]{Rybakov15}
\bysame, \emph{On classification of groups of points on abelian varieties over
  finite fields}, Mosc. Math. J. \textbf{15} (2015), no.~4, 805--815.
  \MR{3438835}

\bibitem[Spr21]{Springer21}
Caleb Springer, \emph{The structure of the group of rational points of an
  abelian variety over a finite field}, European Journal of Mathematics (2021),
  1--13.

\bibitem[Tat66]{Tate66}
John Tate, \emph{Endomorphisms of abelian varieties over finite fields},
  Invent. Math. \textbf{2} (1966), 134--144. \MR{0206004}

\bibitem[Tat71]{Tate71}
\bysame, \emph{Classes d'isog\'{e}nie des vari\'{e}t\'{e}s ab\'{e}liennes sur
  un corps fini (d'apr\`es {T}. {H}onda)}, S\'{e}minaire {B}ourbaki. {V}ol.
  1968/69: {E}xpos\'{e}s 347--363, Lecture Notes in Math., vol. 175, Springer,
  Berlin, 1971, pp.~Exp. No. 352, 95--110. \MR{3077121}

\bibitem[TVN07]{TVN07}
Michael Tsfasman, Serge Vl\u{a}du\c{t}, and Dmitry Nogin, \emph{Algebraic
  geometric codes: basic notions}, Mathematical Surveys and Monographs, vol.
  139, American Mathematical Society, Providence, RI, 2007. \MR{2339649}

\bibitem[Vol88]{Vol88}
J.~F. Voloch, \emph{A note on elliptic curves over finite fields}, Bull. Soc.
  Math. France \textbf{116} (1988), no.~4, 455--458 (1989). \MR{1005390}

\bibitem[Wat69]{Wat69}
William~C. Waterhouse, \emph{Abelian varieties over finite fields}, Ann. Sci.
  \'Ecole Norm. Sup. (4) \textbf{2} (1969), 521--560. \MR{0265369}

\bibitem[Wei48]{Weil48}
Andr\'{e} Weil, \emph{Sur les courbes alg\'{e}briques et les vari\'{e}t\'{e}s
  qui s'en d\'{e}duisent}, Actualit\'{e}s Scientifiques et Industrielles
  [Current Scientific and Industrial Topics], No. 1041, Hermann et Cie., Paris,
  1948, Publ. Inst. Math. Univ. Strasbourg {{\bf{7}}} (1945). \MR{0027151}

\bibitem[Xin94]{Xing94}
Chaoping Xing, \emph{The structure of the rational point groups of simple
  abelian varieties of dimension two over finite fields}, Arch. Math. (Basel)
  \textbf{63} (1994), no.~5, 427--430. \MR{1300737}

\bibitem[Xin96]{Xing96}
\bysame, \emph{On supersingular abelian varieties of dimension two over finite
  fields}, Finite Fields Appl. \textbf{2} (1996), no.~4, 407--421. \MR{1409453}

\end{thebibliography}
\newcommand{\etalchar}[1]{$^{#1}$}
\def\cprime{$'$}
\providecommand{\bysame}{\leavevmode\hbox to3em{\hrulefill}\thinspace}
\providecommand{\MR}{\relax\ifhmode\unskip\space\fi MR }
\providecommand{\MRhref}[2]{%
  \href{http://www.ams.org/mathscinet-getitem?mr=#1}{#2}
}
\providecommand{\href}[2]{#2}

\end{document}